\documentclass[12pt]{amsart}
\usepackage{ifthen,color,xypic,amssymb}

  \newtheorem{thm}{Theorem}[section]
 \newtheorem*{thm*}{Theorem}
 \newtheorem{cor}[thm]{Corollary}
 \newtheorem{lem}[thm]{Lemma}
 \newtheorem{prop}[thm]{Proposition}
 
 \theoremstyle{definition}
 \newtheorem{defin}[thm]{Definition}
 \newenvironment{defn}{\begin{defin}}{\hfill\hspace{1pt}$\triangle$\end{defin}}
 \theoremstyle{remark}
 \newtheorem{rema}[thm]{Remark}
 \newtheorem{exe}[thm]{Example}
 \newenvironment{ex}{\begin{exe}}{\hfill\hspace{1pt}$\triangle$\end{exe}}
\newenvironment{rem}{\begin{rema}}{\hfill\hspace{1pt}$\triangle$\end{rema}}

\newcommand{\D}{\mathcal{D}}

\numberwithin{equation}{section}

\newcommand{\SHom}{{\mathcal{H}om}}
\newcommand{\punt} {{\scriptscriptstyle\bullet}}
\newcommand{\di}{\operatorname{d}}
\DeclareMathOperator{\Coker}{{Coker}}
\newcommand{\cE}{{\mathcal E}}
\newcommand{\dSHom}[1]{{\mathcal{H}om_{#1}^{\scriptscriptstyle\bullet}}}
\newcommand{\Id}{{\operatorname{Id}}}
\newcommand{\dHom}[1]{{\Hom_{#1}^{\scriptscriptstyle\bullet}}}
\newcommand{\lotimes}{{\,\stackrel{\mathbf L}{\otimes}\,}}
\newcommand{\bL}{{\mathbf L}}
\newcommand{\bR}{{\mathbf R}}
\newcommand{\cO}{{\mathcal O}}
\DeclareMathOperator{\Spec}{{Spec}}
\newcommand{\Hom}{{\operatorname{Hom}}}
\newcommand{\bHom}{{ \operatorname{\bf Hom}}}
\begin{document}

\title[Non-exact integral functors]
 {Non-exact integral functors}

\author{Fernando Sancho de Salas}

\address{Departamento de Matem\'{a}ticas e Instituto Universitario de F\'{\i}sica Fundamental y Matem\'{a}ticas (IUFFyM),
Universidad de Salamanca, Plaza de la Merced 1-4, 37008 Salamanca,
Spain}

\email{fsancho@usal.es}





\thanks {Work supported by research projects MTM2009-07289 (MEC)
and GR46 (JCYL)}

\date{\today}

\dedicatory{}



\begin{abstract} We give a natural notion of (non-exact) integral
functor $D_\text{perf}(X)\to D^b_c(Y)$ in the context of
$k$-linear and graded categories. In this broader sense, we prove
that every $k$-linear and graded functor is integral.
\end{abstract}

\maketitle

\section*{Introduction}

Let $k$ be a field, $X$ and $Y$   two projective $k$-schemes and
$K$ an object of $D^b_c(X\times_k Y)$. Let us denote $p\colon
X\times Y\to Y$ and $q\colon X\times Y\to X$ the natural
projections. One has a functor
\[\aligned \Phi_K\colon D_\text{perf}(X)&\to D^b_c(Y),\\ M &\mapsto p_*(K\otimes q^*M )\endaligned
\] This
functor is $k$-linear, graded and exact. We shall say that
$\Phi_K$ is an {\it exact} integral functor of kernel $K$. We have
then a functor
\[ \aligned \Phi\colon D^b_c(X\times Y)&\to \bHom^{\text{ex}}_k(D_\text{perf}(X),D^b_c(Y))\\
K&\mapsto \Phi_K  \endaligned  \] where
$\bHom^{\text{ex}}_k(D_\text{perf}(X),D^b_c(Y))$ denotes the
category of $k$-linear, graded and exact functors (with $k$-linear
and graded natural transformations). This functor is, in general,
neither essentially injective or full  (see \cite{CaSte10-2}) or
faithful (see \cite{Cal}). However, one of the main open questions
is whether it is essentially surjective. One has a positive answer
for fully faithful functors $D_\text{perf}(X)\to D_\text{perf}(Y)$
due to Orlov and Lunts (see \cite{LuOr} and \cite{Or97}). More
generally, A. Canonaco and P. Stellari have shown in
\cite{CaSte07} (see also \cite{CaSte10} for the result in the
supported case) that any exact functor $F\colon D^b(X)\to D^b(Y)$
satisfying
\[ \Hom_{D^b_c(Y)}(F(A), F(B)[k])=0\] for any sheaves $A$ and $B$
on $X$ and any integer $k<0$, is integral. There are also
generalizations of the fully faithful case to derived stacks (see
\cite{Kaw02}) and twisted categories (see \cite{CaSte07}).

The best evidence for a positive answer in general is due to the
results of T\"{o}en concerning dg-categories. Indeed, it is proved in
\cite{To07} that all dg (quasi-)functors between the dg-categories
of perfect complexes on smooth proper schemes are of Fourier-Mukai
type. This result, together with the conjecture by Bondal, Larsen
and Lunts in \cite{BoLarLu} that states that all exact functors
between the bounded derived categories of coherent sheaves on
smooth projective varieties should be liftable to dg
(quasi-)functors between the corresponding dg-enhancements, would
give a positive answer to the question.

In this paper we show that if we work in the context of $k$-linear
and graded categories and $k$-linear and graded functors (i.e., we
forget exactness) the answer is positive. Of course we have to say
what a $k$-linear and graded (may be non exact) integral functor
means. The idea is very simple: since an object $K\in
D^b_c(X\times Y)$ may be thought of as an {\it exact} functor
$D_\text{perf}(X\times Y)\to D_\text{perf}(k)$, we shall instead
consider, as a kernel, a $k$-linear and graded (may be non exact)
functor $\omega\colon D_\text{perf}(X\times Y)\to D (k)$. Let us
be more precise:

For any $k$-scheme  $f\colon Z\to \Spec k$, let us denote
$D_\text{perf}(Z)^*$ the category of $k$-linear and graded
functors $D_\text{perf}(Z)\to D(k)$. One has a natural functor, $
D^b_c(Z)\to D_\text{perf}(Z)^*$, $K\mapsto \omega_K$, where
$\omega_K$ is the exact integral functor of kernel $K$, i.e.,
$\omega_K(M)=f_*(K\otimes M)$.
 This functor is fully faithful and its essential image is the
subcategory $D_\text{perf}(Z)^\vee$ of exact and perfect functors
(perfect means that it takes values in $D_\text{perf}(k)$).


Now, let $\omega\in D_\text{perf}(X\times Y)^*$. It induces, in
the obvious way, a $k$-linear and graded functor
\[\aligned \Phi_\omega\colon D_\text{perf}(X)&\to D_\text{perf}(Y)^*\\
M &\mapsto \Phi_\omega(M)
\endaligned,\]
i.e.,  $\Phi_\omega (M)(N)=\omega (q^*M\otimes p^*N)$ (see also
\eqref{integralfunctor} for an alternative description which is
closer to the usual definition). We shall say that $\Phi_\omega$
is an integral functor of kernel $\omega$. If $\omega$ is exact
and perfect, then $\omega\simeq \omega_K$ for an unique $K\in
D^b_c(X\times Y)$. Then $\Phi_\omega$ takes values in
$D_\text{perf}(Y)^\vee\simeq D^b_c(Y)$ and $\Phi_\omega\simeq
\Phi_K$.

One has then a functor
\[ \aligned \Phi\colon D_\text{perf}(X\times Y)^*&\to \bHom_k (D_\text{perf}(X),D_\text{perf}(Y)^*)\\
\omega&\mapsto \Phi_\omega  \endaligned  \] extending the functor
$\Phi\colon D^b_c(X\times Y) \to \bHom^{\text{ex}}_k
(D_\text{perf}(X),D^b_c(Y) )$.

The aim of this paper is to prove that $\Phi\colon
D_\text{perf}(X\times Y)^* \to \bHom_k (D_\text{perf}(X),
D_\text{perf}(Y)^*)$ is essentially surjective. Even more, we
shall construct a right inverse of $\Phi$, i.e., a functor
$$\Psi\colon \bHom_k (D_\text{perf}(X), D_\text{perf}(Y)^*)\to
D_\text{perf}(X\times Y)^* $$ such that $\Phi\circ\Psi$ is
isomorphic to the identity. In other words, for any $k$-linear and
graded functor $F\colon D_\text{perf}(X)\to D_\text{perf}(Y)^*$
there exist  a kernel $\omega_F\in D_\text{perf}(X\times Y)^*$ and
an isomorphism $F\simeq \Phi_{\omega_F}$ which are functorial on
$F$. This will be a consequence of an extension theorem (Theorem
\ref{lifting}) that states that if $F\colon D_\text{perf}(X)\to
D_\text{perf}(Y)^*$ is a $k$-linear and graded functor and $S$ is
any $k$-scheme, then $F$ can be lifted to an $S$-linear functor
$F_S\colon D_{\text{fhd}/X,S}(X\times S )\to D_\text{perf}(Y\times
S)^*$, where $D_{\text{fhd}/X,S}(X\times S)$ is the category of
objects in $D^b_c(X\times S)$ of finite homological dimension over
both $X$ and $S$ (see Definition \ref{fhd}).

Let us denote $D_\text{perf}(X\times Y)^{Y-\vee}$ the full
subcategory of $D_\text{perf}(X\times Y)^*$ whose objects are the
$\omega\in D_\text{perf}(X\times Y)^*$ which are exact and perfect
on $Y$, i.e. such that for any $M\in D_\text{perf}(X)$,
$\Phi_\omega (M)$ belongs to $D_\text{perf}(Y)^\vee$. Taking into
account the equivalence $D^b_c(Y)\overset\sim\to
D_\text{perf}(Y)^\vee$, we obtain functors $\Phi\colon
D_\text{perf}(X\times Y)^{Y-\vee} \to \bHom_k (D_\text{perf}(X),
D^b_c(Y))$ and $\Psi\colon \bHom_k (D_\text{perf}(X), D^b_c(Y))\to
D_\text{perf}(X\times Y)^{Y-\vee}$, such that $\Phi\circ\Psi$ is
isomorphic to the identity. Finally, if we denote
$D_\text{perf}(X\times Y)^{{\rm bi}-\vee}$ the full subcategory of
$D_\text{perf}(X\times Y)^*$ whose objects are the $\omega\in
D_\text{perf}(X\times Y)^*$ which are bi-exact (see Definition
\ref{def-biexact}), then we obtain functors $\Phi\colon
D_\text{perf}(X\times Y)^{{\rm bi}-\vee} \to \bHom_k^\text{ex}
(D_\text{perf}(X), D^b_c(Y))$ and $\Psi\colon \bHom_k^\text{ex}
(D_\text{perf}(X), D^b_c(Y))\to D_\text{perf}(X\times Y)^{{\rm
bi}-\vee}$, such that $\Phi\circ\Psi$ is isomorphic to the
identity.

We shall also give these results in the relative setting. That is,
assume that $X$ and $Y$ are flat $T$-schemes and let $p\colon
X\times_TY\to Y$, $q\colon X\times_TY\to X$ be the natural
projections. For each $M\in D^b_c(X\times_TY)$ one has a
$T$-linear functor $\Phi_K\colon D_\text{perf}(X)\to D^b_c(Y)$,
$M\mapsto p_*(K\otimes q^*M)$; one has  then a $T$-linear functor
\[ \Phi\colon D^b_c(X\times_TY)\to \bHom_T^\text{ex} (D_\text{perf}(X),
D^b_c(Y))\] More generally, for any $\omega\in
D_\text{perf}(X\times_TY)^*$ one has a $T$-linear functor
$\Phi_\omega\colon D_\text{perf}(X)\to D_\text{perf}(Y)^*$; one
has  then a $T$-linear functor
\[ \Phi\colon D_\text{perf}(X\times_TY)^*\to \bHom_T  (D_\text{perf}(X),
D_\text{perf}(Y)^*).\] As before, we shall construct a $T$-linear
functor
\[ \Psi\colon \bHom_T (D_\text{perf}(X), D_\text{perf}(Y) ^*)\to
D_\text{perf}(X\times_T Y)^*\] such that $\Phi\circ\Psi$ is
isomorphic to the identity.

\section{Notations and basic results}

Throughout the paper $k$ denotes a field. All the schemes are
assumed to be proper $k$-schemes. If $f\colon X\to Y$ is a
morphism of $T$-schemes, we shall still denote by $f$ the morphism
$X\times_TT'\to Y\times_TT'$ induced by $f$ after a base change
$T'\to T$.

For a scheme $X$, we denote by $D(X)$ the derived category of
complexes of $\cO_X$-modules with quasi-coherent cohomology,
$D^b_c(X)$ the full subcategory of complexes with bounded and
coherent cohomology and $D_\text{perf}(X)$ the full subcategory of
perfect complexes.

Since we shall deal with derived categories, we shall use the
abbreviated notations $f_*, f^*,\otimes, \dots$ for the derived
functors $\bR f_*,\bL f^*,\overset\bL\otimes,\dots$.

We shall use extensively the following results of derived
categories:
\begin{enumerate}
\item Projection formula: If $f\colon X\to Y$ is a morphism of
schemes, then one has a natural isomorphism $f_*(M\otimes
f^*L)\simeq (f_*M)\otimes L$, with $M\in D(X)$, $L\in D(Y)$.
\item Flat base change: Let us consider a cartesian diagram
\[\xymatrix{ X\times_TY \ar[r]^p\ar[d]^q & Y\ar[d]^q \\ X\ar[r]^p
& T}\] with $p$ flat. For any $N\in D(Y)$ one has a natural
isomorphism $p^*q_*N\overset\sim\to q_*p^*N$.
\end{enumerate}


\subsection{Non-exact integral functors}$\,$\medskip

Let $k$ be a field, $D(k)$ the derived category of complexes of
$k$-vector spaces.

\begin{defn} Let $p\colon Z\to \Spec k$ be a $k$-scheme. A {\it linear form} on $Z$
is a $k$-linear and graded functor $\omega\colon
D_\text{perf}(Z)\to D(k)$. We say that $\omega $ is {\it perfect}
if it takes values in $D_\text{perf}(k)$. We say that $\omega$ is
exact if it takes exact triangles into exact triangles.

A linear morphism $\omega\to\omega'$ between linear forms on $Z$
is just a morphism of $k$-linear and graded functors. We shall
denote by $D_\text{perf}(Z)^*$ the category of $k$-linear forms on
$Z$ and $k$-linear morphisms and by $D_\text{perf}(Z)^\vee$ the
full subcategory of $D_\text{perf}(Z)^*$ whose objects are the
exact and perfect linear forms on $Z$. Both $D_\text{perf}(Z)^*$
and $D_\text{perf}(Z)^\vee$ are $k$-linear and graded categories.
\end{defn}

For any $K\in D^b_c(Z)$ one has a (perfect and exact) linear form
$\omega_K$ on $Z$, defined by $\omega_K(M)=p_*(K\otimes M)$.
Moreover one has the following

\begin{prop}\label{Fourier} Let $Z$ be a projective $k$-scheme. The functor
$D^b_c(Z)\to D_\text{\rm perf}(Z)^\vee$, $K\mapsto
\omega_K$, is an equivalence (of $k$-linear and graded
categories).
\end{prop}

\begin{proof} It is proved in \cite{BoVdB03} that any
contravariant cohomological functor of finite type over
$D_\text{perf}(Z)$ ($Z$ a projective scheme over $k$) is
representable by a bounded complex with coherent homology. It
follows that if $\omega\colon D_\text{perf}(Z)\to
D_\text{perf}(k)$ is exact, then it has a  right pseudo adjoint
$\omega^\#\colon D_\text{perf}(k)\to D^b_c(Z)$; that is, one has
\[ \Hom_{D(k)}(\omega(M),E) \simeq \Hom_{D(Z)}(M,\omega^\#(E))\] for any
$M\in D_\text{perf}(Z)$, $E\in D_\text{perf}(k)$. Since
$\omega^\#$ is $k$-linear and graded, one has $\omega^\#(E) \simeq
\omega^\#(k)\otimes p^*E$, and then $\omega\simeq \omega_K$ with
$K=\bR\SHom (\omega^\#(k),p^!k)$. Conclusion follows (see
\cite{SaSa} for further details and a more general statement).
\end{proof}

\begin{defn} {\it Tensor product,   direct and inverse image}.
\begin{enumerate}
\item $D_\text{perf}(Z)^*$ has a $D_\text{perf}(Z)$-module
structure: for any $M\in D_\text{perf}(Z)$, $\omega\in
D_\text{perf}(Z)^*$, we define $\omega\otimes M\in
D_\text{perf}(Z)^*$ by the formula $(\omega\otimes
M)(N)=\omega(M\otimes N)$.

\item For any morphism of $k$-schemes $f\colon Z\to Z'$, we define
$f_*\colon D_\text{perf}(Z)^*\to D_\text{perf}(Z')^*$ as the
$k$-linear and graded functor induced by $f^*\colon
D_\text{perf}(Z')\to D_\text{perf}(Z)$; that is,
$(f_*\omega)(M')=\omega(f^*M')$, for $M'\in D_\text{perf}(Z')$,
$\omega \in D_\text{perf}(Z)^*$.
\item If $f\colon Z\to Z'$ is flat, we define $f^*\colon D_\text{perf}(Z')^*\to D_\text{perf}(Z)^*$ as the
$k$-linear and graded functor induced by $f_*\colon
D_\text{perf}(Z)\to D_\text{perf}(Z')$; that is,
$(f^*\omega)(M)=\omega(f_*M)$, for $M \in D_\text{perf}(Z )$,
$\omega \in D_\text{perf}(Z')^*$.
\end{enumerate}
\end{defn}

\begin{defn}Let $p\colon X\to\Spec k$ and $q\colon Y\to\Spec k$ be two
$k$-schemes and $p\colon X\times Y\to Y$, $q\colon X\times Y\to X$
the natural projections. For each $\omega\in D_\text{perf}(X\times
Y)^*$ we have a $k$-linear and graded functor
\begin{equation}\label{integralfunctor}\aligned \Phi_\omega\colon D_\text{perf}(X)&\to
D_\text{perf}(Y)^*\\ M &\mapsto \Phi_\omega(M)=p_*(\omega\otimes
q^*M)
\endaligned\end{equation}
We say that $\Phi_\omega$ is an integral functor of kernel
$\omega$.
\end{defn}

\begin{ex} Assume that $X$ and $Y$ are projective $k$-schemes. If $\omega$ is exact and perfect,
then $\omega\simeq \omega_K$ for a unique $K\in D^b_c(X\times Y)$
by Proposition \ref{Fourier}. Then $\Phi_\omega$ takes values in
$D_\text{perf}(Y)^\vee\simeq D^b_c(Y)$ and $\Phi_\omega$ is
isomorphic to the usual exact integral functor $\Phi_K$.
\end{ex}

Let us denote $\bHom_k(D_\text{perf}(X), D_\text{perf}(Y)^*)$ the
category of $k$-linear and graded functors from $D_\text{perf}(X)$
to $D_\text{perf}(Y)^*$ and $k$-linear and graded morphisms of
functors. It is a $k$-linear and graded category in the obvious
way. One has a $k$-linear and graded functor
\[\aligned \Phi\colon D_\text{perf}(X\times Y)^* &\to \bHom_k(D_\text{perf}(X),
D_\text{perf}(Y)^*)\\ \omega &\mapsto \Phi_\omega
\endaligned\] and a commutative diagram

$$ \xymatrix{  D_\text{perf}(X\times Y)^*   \ar[r]^{\Phi\qquad\quad}   &
  \bHom_k(D_\text{perf}(X), D_\text{perf}(Y)^*)\\    D^b_c(X\times
Y)  \ar[u]  \ar[r]^{\Phi\qquad\quad}   &
\bHom_k^\text{ex}(D_\text{perf}(X), D^b_c(Y)) \ar[u] }$$
 whose
vertical maps are fully faithful.

\section{Main results}

The aim of this section is to construct a functor $\Psi\colon
\bHom_k(D_\text{perf}(X), D_\text{perf}(Y)^*)\to
D_\text{perf}(X\times Y)^* $ such that $\Phi\circ\Psi$ is
isomorphic to the identity. This will be a consequence of the
following extension theorem: a $k$-linear and graded functor
$F\colon D_\text{perf}(X)\to D_\text{perf}(Y)^*$  can be
(functorially) extended to an $S$-linear functor $F_S\colon
D_{\text{fhd}/X,S}(X\times S )\to D_\text{perf}(Y\times S)^*$, for
any $k$-scheme $S$ (see Definition \ref{fhd} for the meaning of
$D_{\text{fhd}/X,S}(X\times S )$).

\subsection{Cokernels of linear forms}$\,$\medskip

Let $E_\punt=\{ E_1 \overset{\dsize{\overset {\di_0}
\longrightarrow}}{ \dsize{\underset{\di_1}\longrightarrow} }
E_0\}$ be two maps in $D(k)$. We define $\Coker (E_\punt)$ as the
cokernel of the morphisms of vector spaces $
H(E_1)\overset{\dsize{\overset {H(\di_0)} \longrightarrow}}{
\dsize{\underset{H(\di_1)}\longrightarrow} } H(E_0)$.

%


Now let $p\colon Z \to \Spec k$ be a  $k$-scheme and let
$\omega_\punt = \{ \omega_1 \overset{\dsize{\overset {\di_0}
\longrightarrow}}{ \dsize{\underset{\di_1}\longrightarrow} }
\omega_0\}$ be two morphisms in $D_\text{perf}(Z)^*$. For each
$M\in D_\text{perf}(Z)$, let us denote $\omega_\punt(M) =\{
\omega_1(M) \overset{\dsize{\overset {\di_0(M)} \longrightarrow}}{
\dsize{\underset{\di_1(M)}\longrightarrow} } \omega_0(M)\}$. We
define $\Coker (\omega_\punt)$ as the object in
$D_\text{perf}(Z)^*$ defined by
\[ \Coker (\omega_\punt)(M)= \Coker (\omega_\punt(M))\]

%

This is clearly functorial on $\omega_\punt$. The next proposition
is immediate.

\begin{prop} Let $\omega_\punt = \{ \omega_1 \overset{\dsize{\overset {\di_0}
\longrightarrow}}{ \dsize{\underset{\di_1}\longrightarrow} }
\omega_0\}$ be two morphisms in $D_\text{\rm perf}(Z)^*$.
\begin{enumerate}
\item For any $M\in D_\text{perf}(Z)$ one has $\Coker(\omega_\punt \otimes M)=
\Coker(\omega_\punt) \otimes M$.
\item For any morphism $f\colon Z\to Z'$, one has
$f_*\Coker(\omega_\punt)=\Coker(f_*\omega_\punt)$.
\end{enumerate}

\end{prop}

\medskip
\subsection{\bf The perfect-resolution.}\label{section-resolution}\hfill
\medskip

Let $p\colon X\to \Spec k$, $f\colon S\to \Spec k$ be two
$k$-schemes and $M\in D^b_c(X\times S)$ an object of finite
homological dimension over $S$ (see Definition \ref{fhd}). We
still denote by $p\colon X\times S\to S$ and $f\colon X\times S\to
X$ the natural projections. For each $\cE\in D_\text{perf}(X)$ we
shall denote $R_\cE(M)=  f^*\cE \otimes p^*p_* (f^*\cE^*\otimes
M)$, with $\cE^*=\bR\dSHom{X}(\cE,\cO_X)$. One has a natural
morphism $\rho_M^\cE\colon R_\cE(M)\to M$. We shall denote
\[ R_0(M)=\underset {\cE\in D_\text{perf}(X)} \oplus R_\cE(M)\] and $\rho_M\colon R_0(M)\to
M$ the natural map. This is functorial on $M$.

Let $R_1(M)=R_0(R_0(M))$. One has two morphisms
\[ R_1(M)\overset{\dsize{\overset {\di_0}
\longrightarrow}}{ \dsize{\underset{\di_1}\longrightarrow} }
R_0(M)\] namely: $\di_0=\rho_{R_0(M)}$ and $\di_1=R_0(\rho_M)$. It
is immediate to check that $\rho_M\circ\di_0 = \rho_M\circ\di_1$.

More explicitly, $R_1(M)=\underset{\cE_1, \cE_0\in
D_\text{perf}(X)}\oplus R_{\cE_1} R_{\cE_0}(M)$ and
\begin{equation}\label{R_1} R_{\cE_1} R_{\cE_0}(M)\simeq
f^*\cE_1\otimes p^*p_*f^*(\cE_1^*\otimes \cE_0)\otimes
p^*p_*(f^*\cE_0^*\otimes M).\end{equation}

The differentials $\di_0,\di_1\colon R_1(M)\to R_0(M)$ are induced
by the morphisms $$\aligned \rho_{R_{\cE_0}(M)}^{\cE_1}&\colon
R_{\cE_1} R_{\cE_0}(M)\to R_{\cE_0}(M) \\
R_{\cE_1}(\rho_M^{\cE_0}) &\colon R_{\cE_1} R_{\cE_0}(M)\to
R_{\cE_1}(M)\endaligned .$$

\begin{prop}\label{resolution} For any $L\in D_\text{\rm perf}(X)$,  $R_1(f^*L)\overset{\dsize{\overset {\di_0}
\longrightarrow}}{ \dsize{\underset{\di_1}\longrightarrow} }
R_0(f^*L)\overset{\rho_{f^*L}} \longrightarrow f^*L$
 is exact.
\end{prop}

\begin{proof}  The morphism $\rho_{f^*L}^{f^*L}\colon R_{f^*L}(f^*L)\to f^*L$ has a
natural section $h\colon f^*L \to R_{f^*L}(f^*L)=  f^*L \otimes
p^*p_* (f^*L^*\otimes f^*L)$ induced by the natural map $\cO_S\to
p_* (f^*L^*\otimes f^*L)$. Then we have a map $h_0\colon f^*L\to
R_0(f^*L)$ which is $h$ in the $f^*L$-component and zero in the
others. It is clear that $h_0$ is a section of $\rho_{f^*L}$. Now
define $h_1\colon R_0(f^*L)\to R_1(f^*L)$ as $h_1=-R_0(h)$. One
can check that $(\di_0-\di_1)\circ h_1+h_0\circ \rho_{f^*L}=\Id$,
 hence the result.
\end{proof}

We shall denote $R^{X\times S/S}_\punt(M):=\{
R_1(M)\overset{\dsize{\overset {\di_0} \longrightarrow}}{
\dsize{\underset{\di_1}\longrightarrow} } R_0(M)\}$.

\begin{rem}\label{remark-resolution} If $G\colon D_\text{perf}(X\times S)\to \D$ is an additive
functor and $\D$ has infinite direct sums, we can define
$G(R_1(M)):=\underset{\cE_1\cE_0}\oplus G(R_{\cE_1}R_{\cE_0}(M))$
and $G(R_0(M)):=\underset{\cE}\oplus G(R_\cE(M))$. One has
differentials $G(\di_0), G(\di_1)\colon G(R_1(M))\to G(R_0(M))$.
We shall denote $$G(R_\punt(M))=\{ G(
R_1(M))\overset{\dsize{\overset {G(\di_0)} \longrightarrow}}{
\dsize{\underset{G(\di_1)}\longrightarrow} } G(R_0(M))\}.$$ For
any $L\in D_\text{\rm perf}(X)$,
$G(R_1(f^*L))\overset{\dsize{\overset {G(\di_0)}
\longrightarrow}}{ \dsize{\underset{G(\di_1)}\longrightarrow} }
G(R_0(f^*L))\overset{G(\rho_{f^*L})} \longrightarrow G(f^*L)$ is
exact.
\end{rem}

In the next propositions we shall prove the $S$-linearity of
$R_\punt^{X\times S/S}(M)$ and its compatibility with direct
images.

\begin{prop}\label{linearity} For any $V\in D_\text{\rm perf}(S)$ one has a
natural isomorphism $R_\punt^{X\times S/S}(M\otimes
p^*V)\overset\sim\to R_\punt^{X\times S/S}(M)\otimes p^*V$.
\end{prop}

\begin{proof} One has a natural isomorphism $R_\cE(M\otimes p^*V)\simeq R_\cE(M)\otimes
p^*V$. Indeed, by projection formula,

\[ \aligned R_\cE(M\otimes p^*V)=  f^*\cE \otimes p^*p_* (f^*\cE^*\otimes M\otimes
p^*V) & \simeq  f^*\cE \otimes p^*p_*  (f^*\cE^*\otimes M) \otimes
p^*V
\\ & = R_\cE(M)\otimes p^*V.\endaligned
\] One checks that the diagram
$$\xymatrix{  R_\cE (M\otimes p^*V)\ar[rr]^{\rho^\cE_{M\otimes p^*V}}\ar[d]^{\wr}
& &
M\otimes p^*V\ar[d]^{\Id}\\
R_\cE (M)\otimes p^*V \ar[rr]^{\rho^\cE_M\otimes 1}&  &  M\otimes
p^*V }$$ is commutative. Hence one has an isomorphism  $R_0
(M\otimes p^*V)\overset\sim\to R_0 (M)\otimes p^*V$ and a
commutative diagram

$$\xymatrix{  R_0 (M\otimes p^*V)\ar[rr]^{\rho_{M\otimes p^*V}}\ar[d]^{\wr}
& &
M\otimes p^*V\ar[d]^{\Id}\\
R_0 (M)\otimes p^*V \ar[rr]^{\rho_M\otimes 1}&  &  M\otimes p^*V
}$$ Conclusion follows.
\end{proof}

\begin{prop}\label{compatibility} One has a natural isomorphism $f_*[R_\punt^{X\times
S/S}(M) ]\simeq R_\punt^{X /k}(f_*M)$.
\end{prop}

\begin{proof}
By projection formula and flat base change one has
$$\aligned f_*R_\cE(M)=f_*[f^*\cE  \otimes   p^*p_* (f^*\cE^*\otimes M)]
\simeq \cE \otimes f_* p^*p_* (f^*\cE^*\otimes M) &\simeq \cE
\otimes p^*p_* ( \cE^*\otimes f_*M)\\ & = R_\cE(f_*M)\endaligned$$
Moreover, the diagram
$$\xymatrix{ f_* R_\cE (M )\ar[rr]^{f_*(\rho^\cE_M)}\ar[d]^{\wr}
& &
M \ar[d]^{\Id}\\
R_\cE(f_*M)\otimes p^*V \ar[rr]^{\rho^\cE_{f_*M}}&  & M  }$$ is
commutative. One has then $f_*R_0(M)\simeq R_0(f_*M)$ and a
commutative diagram
$$\xymatrix{ f_* R_0 (M )\ar[rr]^{f_*(\rho_M)}\ar[d]^{\wr}
& &
M \ar[d]^{\Id}\\
R_0 (f_*M)\otimes p^*V \ar[rr]^{\rho_{f_*M}}&  & M  }$$ Conclusion
follows.
\end{proof}

\subsection{The extension theorem.}\hfill
\medskip

We need to introduce a relative notion of perfectness.
\begin{defn}\label{fhd} Let $f\colon Z\to T$ be a morphism of schemes. An
object $M\in D^b_c(Z)$ is said to be of finite homological
dimension over $T$ (fhd$/T$ for short), if $M\otimes f^*N$ is
bounded and coherent for any $N\in D^b_c(T)$. We shall denote by
$D_{\text{fhd}/T}(Z)$ the faithful subcategory of the objects of
finite homological dimension over $T$.
\end{defn}

The following properties of fhd-objects are quite immediate (see
\cite[Section 1.2]{HLS05}).

\begin{prop}\label{fhd-properties}\begin{enumerate}
\item If $f$ is flat, then $D_\text{\rm perf}(Z)\subset
D_{\text{\rm fhd}/T}(Z)$.
\item If $M$ is fhd over $T$  and $f$ is proper, then $f_*M$ is perfect.
\item If $M$ is fhd over $T$ and $\cE\in D_\text{\rm perf}(Z)$, then
$M\otimes\cE$ is fhd over $T$.
\end{enumerate}
\end{prop}

Given two schemes $X$ and $S$, we shall denote by
$D_{\text{fhd}/X,S}(X\times S)$ the category of objects in
$D^b_c(X\times S)$ of finite homological dimension over both $X$
and $S$.

\begin{thm}\label{lifting} Let $p\colon X\to\Spec k$ and $q\colon Y\to\Spec k$ be two proper $k$-schemes.
Let $F\colon D_\text{\rm perf}(X)\to D_\text{\rm perf}(Y)^*$ be a
$k$-linear and graded  functor. For any proper $k$-scheme $f\colon
S\to\Spec k$ there exists a functor
\[ F_S\colon D_{\text{\rm fhd}/X,S}(X\times S)\to D_\text{\rm perf}(Y\times S)^*\] such that:

1) $F_S$ is $S$-linear: one has a bi-functorial isomorphism
$F_S(M\otimes p^*V)\simeq F_S(M)\otimes q^*V$ for any $M\in
D_{\text{\rm fhd}/X,S}(X\times S)$, $V\in D_\text{\rm perf}( S)$.

2) It is compatible with $F$: for any $M\in D_{\text{\rm
fhd}/X,S}(X\times S)$ one has a natural isomorphism
$f_*F_S(M)\simeq F(f_*M)$.

3) $F_S$ is functorial on $F$.

\end{thm}

\begin{proof} Let $M\in D_{\text{fhd}/X,S}(X\times S)$. Let us take the
``perfect'' resolution of $M$, $R_\punt^{X\times S/S}(M)$,
constructed in section \ref{section-resolution}. Recall that
$R_\cE(M)=
 f^*\cE \otimes p^*p_*(f^*\cE^*\otimes M)$. Since $M$ is fhd over
$S$, $R_\cE(M)$ belongs to $p^*D_\text{perf}(S)\otimes
f^*D_\text{perf}(X)$. In particular, $R_\cE(M)$ is fhd over $S$
and then $R_{\cE_1}R_{\cE_0}(M)$ belongs also to
$p^*D_\text{perf}(S)\otimes f^*D_\text{perf}(X)$.

Taking in mind the explicit expression of $R_\cE(M)$ and
$R_{\cE_1}R_{\cE_0} (M)$ (see \eqref{R_1}), let us put
$$\aligned \tilde F (R_\cE(M))&: = f^*F(\cE )\otimes
q^*p_*(f^* \cE^*\otimes M)\\ \tilde F (R_{\cE_1} R_{\cE_0}(M))&: =
f^*F(\cE_1)\otimes q^*p_*f^*(\cE_1^*\otimes \cE_0)\otimes
q^*p_*(f^*\cE_0^*\otimes M)  \endaligned$$ and $$\aligned \tilde
F(R_0(M)) &:= \underset{\cE \in D_\text{perf}(X)}\oplus \tilde F ( R_\cE (M))\\
\tilde F (R_1(M))&:= \underset{\cE_1, \cE_0\in
D_\text{perf}(X)}\oplus \tilde F ( R_{\cE_1}
R_{\cE_0}(M))\endaligned$$ By Lemma \ref{lemita} (see below) the
morphisms $\di_0,\di_1\colon R_1(M)\to R_0(M)$ induce   morphisms
$$\tilde\di_0,\tilde\di_1\colon \tilde F (R_1(M))\to \tilde F
(R_0(M))$$ in $D_\text{perf}(Y\times S)^*$ which are functorial on
$M$. Finally, we define
\[ F_S(M):= \Coker (\tilde F(R_\punt(M))\] It is clear that
$F_S(M)$ is functorial on $M$, hence we obtain  a functor
$F_S\colon D_\text{perf}(X\times S)\to D_\text{perf}(Y\times
S)^*$. By construction, $F_S$ satisfies 3).

\begin{lem}\label{lemita} Let $V,V'\in D_\text{\rm perf}(S)$ and $\cE,\cE'\in D_\text{\rm perf}(X)$.
One has a natural map
\[\Hom_{D_\text{\rm perf}(X\times S)} (p^*V\otimes f^*\cE,p^*V'\otimes f^*\cE')\to
\Hom_{D_\text{\rm perf}(Y\times S)^*}
 ( q^*V\otimes f^*F(\cE), q^*V'\otimes
f^*F(\cE'))\] This map is compatible with composition; moreover,
it is $S$-linear and extends $F$ (the precise meaning of these
will be given in the proof).
\end{lem}

\begin{proof}

A morphism $h\colon p^*V\otimes f^*\cE\to p^*V'\otimes f^*\cE'$
corresponds with a morphism $\bar h\colon \cE\to \cE'\otimes_k
\bR\dHom{}(V,V')$. Since $F$ is $k$-linear and graded, it  induces
a morphism $F(\bar h)\colon F(\cE)\to F(\cE')\otimes_k
\bR\dHom{}(V,V')$. Hence, for any $N\in D_\text{perf}(Y\times S)$
one has morphisms
\[\aligned   (q^*V&\otimes f^*F(\cE))(N) =
F(\cE)(f_*(q^*V\otimes N))\overset{F(\bar h)} \to
[F(\cE')\otimes_k \bR\dHom{}(V,V')] (f_*(q^*V\otimes N)) \\ &=
F(\cE') (f_*(q^*V\otimes N)\otimes_k \bR\dHom{}(V,V'))
\overset{(*)}\to   F(\cE')(f_*(q^*V\otimes N))\\ &= (q^*V'\otimes
f^*F(\cE'))(N) \endaligned\] where (*) is the morphism induced by
the natural evaluation map $f_*(q^*V\otimes N)\otimes_k
\bR\dHom{}(V,V')\to f_*(q^*V'\otimes N)$. That is, one obtains a
morphism $\tilde h\colon q^*V\otimes f^*F(\cE)\to q^*V'\otimes
f^*F(\cE')$.

One can check from the construction that $\widetilde{(f\circ
g)}=\tilde f\circ \tilde g$, for any $f\colon p^*V'\otimes
f^*\cE'\to p^*V''\otimes f^*\cE''$ and $g\colon p^*V\otimes
f^*\cE\to p^*V'\otimes f^*\cE'$.  Moreover, if $f=p^*(f_1)\otimes
f_2$ for some $f_1\colon V\to V'$ and $f_2\colon f^*\cE\to
f^*\cE'$, then  $\tilde f=q^*(f_1)\otimes \tilde {f_2}$. Finally,
if $V=V'$ and $f=\Id\otimes f^*(f_2)$ for some $f_2\colon
 \cE\to \cE'$, then $\tilde f=\Id\otimes f^*F( f_2)$.\end{proof}

To conclude the proof of Theorem \ref{lifting}, we have to prove
that $F_S$ satisfies 1) and 2).

\begin{prop}\label{tilde} One has natural isomorphisms:

a) $\tilde F(R_\punt (M))\otimes q^*N  \simeq \tilde F(R_\punt  (M
\otimes p^*N))$ and

b) $f_* \tilde F( R_\punt  (M))\simeq F(f_* R_\punt (M)) $ (see
Remark \ref{remark-resolution} for the definition of $F(f_*
R_\punt (M))$).
\end{prop}

\begin{proof} a) Completely analogous arguments to that of
Proposition \ref{linearity} yield isomorphisms $ \tilde F (R_\cE(M
\otimes p^* V) \simeq  \tilde F (R_\cE(M)) \otimes q^* V$ and $
\tilde F (R_{\cE_1}R_{\cE_0}(M  \otimes p^* V)) \simeq  \tilde F
(R_{\cE_1}R_{\cE_0}(M)) \otimes q^* V$. One checks that these
isomorphisms are compatible with the differentials.

b) Completely analogous arguments to that of Proposition
 \ref{compatibility} yield isomorphisms $f_* \tilde F( R_\cE  (M))\simeq F(f_*
R_\cE  (M))$ and $f_* \tilde F( R_{\cE_1}R_{\cE_0} (M))\simeq
F(f_* R_{\cE_1}R_{\cE_0} (M))$. Again, one checks that these
isomorphisms are compatible with the differentials.
\end{proof}

It follows immediately that $F_S$ satisfies 1). For 2), one has
\[\aligned  f_*F_S(M)=f_*\Coker(\tilde F(R_\punt (M)))= \Coker(f_*\tilde F(R_\punt
(M))) & \overset{\ref{tilde}} \simeq \Coker F(f_*R_\punt (M))\\
& \overset{\ref{compatibility}}\simeq \Coker F(R_\punt (f_*
M))\endaligned
\] Finally, $\Coker F(R_\punt (f_* M)) \simeq F(f_*M)$ by
Proposition \ref{resolution} and Remark \ref{remark-resolution}.

This concludes the proof of Theorem \ref{lifting}.
\end{proof}

\begin{rem} The lifting $F_S$ of $F$ is functorial but it is not
unique. Let us show an  alternative lifting $F'_S$. Instead of
considering the ``resolution'' $R_1(M)\overset{\dsize\to}\to
R_0(M)$, let us consider the complex of objects in $D(X\times S)$
\[ R_\punt(M):= \{ \cdots \to R_n(M)\to R_{n-1}(M)\to \cdots \to R_1(M)\to R_0(M)\}\]
where $R_n(M)=R_0(R_{n-1}(M))$ and the differential $R_n(M)\to
R_{n-1}(M)$ is the alternate sum of the $n+1$ natural  maps from
$R_n(M)$ to $R_{n-1}(M)$. As in the proof of the theorem, we can
define $\tilde F (R_\punt(M))$, which is a complex of objects in
$D_\text{perf}(Y\times S)^*$. Then one defines $F'_S(M)$ as the
``simple complex'' associated to $\tilde F (R_\punt(M))$, i.e.,
for any $N\in D_\text{perf}(Y\times S)$ we define $F'_S(M)(N)$ as
the simple complex associated to the complex of vector spaces
\[ \cdots \to H(\tilde F(R_n(M))(N))\to H(\tilde
F(R_n(M))(N))\to\cdots\to H(\tilde F(R_0(M))(N))\] This functor
$F'_S$ also satisfies properties 1), 2) and 3). Moreover, it has
an extra ``exact'' property: first notice that $F'_S(M)$ is in
fact a functor from $D_\text{perf}(Y\times S)$ to the category of
complexes of vector spaces (i.e, if $h\colon N\to N'$ is a
morphism in $D_\text{perf}(Y\times S)$, then $F'_S(M)(h)$ is a
morphism  of complexes); the exact property is the following: if
$N_1\to N_2\to N_3$ is an exact triangle in $D_\text{perf}(Y\times
S)$, then
\[ F'_S(M)(N_1)\to F'_S(M)(N_2)\to F'_S(M)(N_3)\] is an exact
sequence of complexes (but may be not an exact triangle).
\end{rem}

Let us see now how the extension theorem yields the integrality
theorem.

\begin{thm}\label{integral1} Let $X$ and $Y$ be two proper $k$-schemes and $F\colon D_\text{\rm perf}(X)
\to D_\text{\rm perf}(Y)^*$   a $k$-linear graded functor. Then
there exists an  object $\omega$ in $D_\text{\rm perf}(X\times
Y)^*$ such that $F\simeq \Phi_\omega$.
\end{thm}

\begin{proof} Let $F_S\colon D_{\text{fhd}/X,S}(X\times S)\to D_\text{perf}(Y\times
S)^*$ be the $S$-linear functor given by Theorem \ref{lifting}.
Take $S=X$, $f=p$, $\delta\colon X\to X\times S$ the diagonal map
and $\cO_\Delta=\delta_*\cO_X$. Notice that $\cO_\Delta$ is fhd
over both $X$ and $S$. Then, by properties 1) and 2) of $F_S$,
\[ F(M)\simeq  F(f_*( \cO_\Delta \otimes p^*M)) \simeq f_*F_S(
\cO_\Delta \otimes p^*M)\simeq f_* (F_S(\cO_\Delta)\otimes q^*M)\]
So it is enough to take $\omega=F_S(\cO_\Delta)$.
\end{proof}

Since $F_X$ is functorial on $F$ we obtain:

\begin{cor} One has a functor
\[\aligned \Psi\colon \bHom_k (D_\text{\rm perf}(X),D_\text{\rm perf}(Y)^*)  &\to
D_\text{\rm perf}(X\times Y)^*\\ F&\mapsto
F_X(\cO_\Delta)\endaligned  \] and the composition $\Phi\circ\Psi$
is isomorphic to the identity.
\end{cor}

\subsection{Exactness.}

\begin{defn} A linear form, $\omega\colon D_\text{perf}(X\times Y)\to D(k)$, on $X\times
Y$ is said to be exact and perfect on $Y$  if for any $M\in
D_\text{perf}(X)$, the functor $\Phi_\omega(M)\colon
D_\text{perf}(Y)\to D(k)$ is exact and perfect, i. e.,
$\Phi_\omega(M)\in D(Y)^\vee$.
\end{defn}

We shall denote $D(X\times Y)^{Y-\vee}$ the full subcategory of
$D(X\times Y)^*$ whose objects are the  linear forms on $X\times
Y$ which are exact and perfect on $Y$.

Taking into account that $D^b_c(Y)\to D_\text{perf}(Y)^\vee$ is an
equivalence ($Y$ projective), we obtain

\begin{cor} If $Y$ is projective, one has functors

\[ \aligned \Phi\colon D_\text{\rm perf}(X\times Y)^{Y-\vee} &\to \bHom_k (D_\text{\rm perf}(X),D^b_c(Y) )  \\
\omega &\mapsto \Phi_\omega
\endaligned\] and \[
\aligned \Psi\colon \bHom_k (D_\text{\rm perf}(X),D^b_c(Y)) &\to D_\text{\rm perf}(X\times Y)^{Y-\vee}\\
F&\mapsto F_X(\cO_\Delta)\endaligned  \] and the composition
$\Psi\circ\Phi$ is isomorphic to the identity.
\end{cor}

A linear form $\omega$ on $X\times Y$ also defines an integral
functor in the opposite direction (i.e., from $Y$ to $X$), which
we shall denote by $\overline\Phi_\omega\colon D_\text{perf}(Y)\to
D_\text{perf}(X)^*$.

\begin{prop}\label{biexact} Assume that $X$ and $Y$ are projective. Let $\omega$ be a linear form on
$X\times Y$. The
following conditions are equivalent:

\begin{enumerate}
\item $\omega $ is   exact and perfect on $Y$ and the functor $\Phi_\omega\colon
D_\text{\rm perf}(X)\to D_\text{\rm perf}(Y)^\vee\simeq D^b_c(Y)$
is exact.
\item $\omega $ is exact and perfect on   $X$ and the functor $\overline \Phi_\omega\colon
D_\text{\rm perf}(Y)\to D_\text{\rm perf}(X)^\vee\simeq D^b_c(X)$
is exact.
\end{enumerate}
\end{prop}

\begin{proof}  For any $M\in D_\text{perf}(X)$, $N\in
D_\text{perf}(Y)$, one has $\Phi_\omega(M)(N)=
\overline\Phi_\omega(N)(M)$. Let us see that (1) $\Rightarrow$
(2).

From the equality $\Phi_\omega(M)(N)= \overline\Phi_\omega(N)(M)$
and (1) it follows immediately that $\omega$ is exact and perfect
on $X$. It remains to prove that $\overline\Phi_\omega\colon
D_\text{perf}(Y)\to D_\text{perf}(X)^\vee\simeq D^b_c(X)$ is
exact.

For each object $\cE\in D_\text{perf}(Y)$, let us denote
$\cE^\#:=\bR\dSHom{Y}(\cE, q^!k)$. The functor $H\colon
D_\text{perf}(X)\to \operatorname{Vect}(k)$ defined by
$H(M)=\Hom_{D(Y)}(\Phi_\omega (M),\cE^\#)$ is a contravariant
cohomological functor of finite type, hence it is representable by
an object $\Phi_\omega^\#(M)\in D^b_c(X)$. Hence we obtain a
pseudo right adjoint of $\Phi_\omega$:
\[\Phi_\omega^\# \colon D_\text{perf}^\# (Y)\to D^b_c(X),\] where
$D_\text{perf}^\# (Y)$ is the full subcategory of $D^b_c(Y)$ whose
objects are of the form $\cE^\#$, with $\cE\in D_\text{perf}(Y)$.
That is, one has
\[ \Hom_{D(Y)}(\Phi_\omega
(M),\cE^\#)=\Hom_{D(X)}(M,\Phi_\omega^\#(\cE^\#))\] for any $M\in
D_\text{perf}(X)$, $\cE\in D_\text{perf}(Y)$. Now, since
$\Phi_\omega$ is exact, $\Phi_\omega^\#$ is also exact (one can
copy the same proof than \cite[Lemma 4.11]{Ballard-1}). Finally,
it is easy to see that the equality $\Phi_\omega(M)(N)=
\overline\Phi_\omega(N)(M)$ implies that
$\overline\Phi_\omega(N)=[ \Phi_\omega^\#(N^\#)]^\#$. Hence
$\overline\Phi_\omega$ is exact.
\end{proof}

\begin{defn}\label{def-biexact} A linear form $\omega\colon D_\text{perf}(X\times Y)\to D(k)$  is
called bi-exact if it satisfies any of the  equivalent conditions
of Proposition \ref{biexact}.
\end{defn}

We denote by $D_\text{perf}(X\times Y)^{\text{bi-}\vee}$ the full
subcategory $D_\text{perf}(X\times Y)^*$  whose objects are the
bi-exact linear forms on $X\times Y$. We have then embeddings
\[ D^b_c(X\times Y)\simeq D_\text{perf}(X\times Y)^\vee\hookrightarrow
 D_\text{perf}(X\times Y)^{\text{bi-}\vee}\hookrightarrow D_\text{perf}(X\times
 Y)^{Y-\vee}\hookrightarrow D_\text{perf}(X\times Y)^*\]

Finally, for bi-exact linear forms we have:

\begin{cor} Assume that $X$ and $Y$ are projective. One has  functors
\[ \aligned \Phi\colon D_\text{\rm perf}(X\times Y)^{\text{\rm bi-}\vee} &\to \bHom_k^{\text{\rm ex}}
(D_\text{\rm perf}(X), D^b_c(Y) )  \\
\omega &\mapsto \Phi_\omega
\endaligned $$ and $$
\aligned \Psi\colon \bHom_k^{\text{\rm ex}}(D_\text{\rm
perf}(X),D^b_c(Y)) &\to
D_\text{\rm perf} (X\times Y)^{\text{bi-}\vee}\\
F&\mapsto F_X(\cO_\Delta)\endaligned  \] and the composition
$\Psi\circ\Phi$ is isomorphic to the identity.
\end{cor}

\section{Relative Integral Functors}

In this section we shall reproduce the main results of the
previous section for relative schemes. Let $p\colon X\to T$ and
$q\colon Y\to T$ be two proper $T$-schemes. Let us still denote by
$p\colon X\times_TY\to Y$ and $q\colon X\times_TY\to Y$ the
natural morphisms. For each object $K\in D^b_c(X\times_TY)$ one
has the (relative) exact integral functor
\[\aligned \Phi_K\colon D_\text{perf}(X)&\to D^b_c(Y)\\ M&\mapsto p_*(K\otimes q^*M) \endaligned\]
This functor is $T$-linear: for any $M\in D_\text{perf}(X)$,
$\cE\in D_\text{perf}(T)$ one has a natural isomorphism
$\Phi_K(M\otimes p^*\cE)\simeq \Phi_K(M)\otimes q^*\cE$.

If we replace $K$ by an object  $\omega\in
D_\text{perf}(X\times_TY)^*$, then we have a functor
\[\aligned \Phi_\omega\colon D_\text{perf}(X)&\to D_\text{perf}(Y)^*\\
M&\mapsto p_*(\omega\otimes q^*M) \endaligned\] which is also
$T$-linear (in a natural sense, see below). This will be called a
(relative non-exact) integral functor. Our aim is to show that
(under flatness hypothesis of $p$ and $q$) any $T$-linear functor
$D_\text{perf}(X) \to D_\text{perf}(Y)^*$ is integral, i.e., it is
isomorphic to $\Phi_\omega$ for some $\omega\in
D_\text{perf}(X\times_TY)^*$.

We shall first give some natural definitions about $T$-linear
categories and functors.

\begin{defn}\label{$T$-linear-cat} Let $T$ be a scheme. A $T$-linear structure on an additive graded
category $\D$ is a biadditive and bigraded functor
\[ \aligned D_\text{perf}(T)\times\D &\to\D\\ (\cE,P) &\mapsto \cE \otimes P  \endaligned
\]
satisfying functorial isomorphisms: \begin{enumerate}\item
$\phi_P\colon \cO_T \otimes P\simeq P $.
\item $\psi_{\cE_1,\cE_2,P}\colon \cE_1\otimes (\cE_2\otimes P) \simeq
(\cE_1\lotimes_{\cO_T} \cE_2)\otimes P$.
\end{enumerate}
\end{defn}

\begin{defn} A $T$-linear category is a graded category endowed with a
$T$-linear structure. A $T$-linear functor $F\colon \D\to\D'$
between $T$-linear categories is a functor  endowed with a
bi-additive and bi-graded bi-functorial isomorphism
$\theta_F(P,E)\colon F(\cE\otimes P)\simeq \cE\otimes F(P)$,
 $\cE\in D_\text{perf}(T)$, $P\in\D$, which is compatible
with $\phi_P$ and $\psi_{\cE_1,\cE_2,P}$ in the obvious sense.
That is, a  $T$-linear functor is a pair $(F,\theta_F)$, though we
shall usually denote it by $F$.

A $T$-linear morphism $\phi\colon F\to F'$ between $T$-linear
functors is a morphism of functors which is compatible with the
$\theta's$, i.e., such that the diagram

$$\xymatrix{ F(\cE\otimes P) \ar[r]^\sim \ar[d]_{\phi(\cE\otimes P)} & \cE\otimes
F(P) \ar[d]^{1\otimes \phi(P) } \\ F'(\cE\otimes P) \ar[r]^\sim &
E\otimes F'(P)  }$$ is commutative.
\end{defn}

Is $S=\Spec k$, the above notion of $k$-linear category coincides
with the usual notion of $k$-linear graded category.

We shall denote by $\bHom_T(\D,\D')$ the category of $T$-linear
functors from $\D$ to $\D'$ and $T$-linear morphisms. It has a
natural $T$-linear structure. If $\D$ and $\D'$ are triangulated
categories, we shall denote by $\bHom_T^\text{ex}(\D,\D')$ the
full subcategory of exact $T$-linear functors.

\begin{exe} If $p\colon X\to T$ is a $T$-scheme, then
$D_\text{perf}(X)$ (or $D(X)$, $D^b_c(X)$) has a natural
$T$-linear structure, namely: $\cE\otimes P:=
p^*\cE\otimes_{\cO_X} P$. If $q\colon Y\to T$ is another
$T$-scheme and $K\in D^b_c(X\times_T Y)$, then the integral
functor $\Phi_K\colon D_\text{perf} (X)\to D_\text{perf}  (Y)$ is
a $T$-linear functor (with the $\theta$ induced by the projection
formula).
\end{exe}

For any $k$-linear category $\D$, we shall denote $\D^{*}=
\bHom_k(\D,D(k))$.  If $\D$ is a triangulated category, we shall
denote $\D^\vee= \bHom_k^\text{ex}(\D,D_\text{perf}(k))$. If $\D$
is a $T$-linear category, then $\D^*$ has a natural $T$-linear
structure, defining $(\cE\otimes\omega) (M)=\omega(\cE\otimes M)$.
Moreover, if $\D$ is triangulated and $\cE\otimes (-)\colon
\D\to\D$ is exact for any $\cE\in D_\text{perf}(T)$, then $D^\vee$
is also a $T$-linear category.

%
%

For any  $T$-scheme $p\colon Z\to T$,  the equivalence
$D^b_c(Z)\overset\sim \to D_\text{perf}(Z)^*$ of Proposition
\ref{Fourier} is $T$-linear ($Z$ is a projective $k$-scheme).

Let $X$ and $Y$ be two $T$-schemes and $\omega\colon
D_\text{perf}(X\times_T Y)\to D(k)$ a $k$-linear form on
$X\times_T Y$. Let us denote $p\colon X\times_T Y\to X$ and
$q\colon X\times_T Y\to Y$ the natural projections. For each $M\in
D_\text{perf}(X)$ we have a $k$-linear functor
\[ \Phi_\omega(M) \colon D_\text{perf}(Y)\to D(k)\] defined by
$\Phi_\omega(M)(N)=\omega(p^*M\otimes q^*N)$. We have then a
$T$-linear functor
\[\Phi_\omega\colon D_\text{perf}(X)\to D_\text{perf}(Y)^*.\]

\begin{defn} We say that $\Phi_\omega\colon D_\text{perf}(X)\to D_\text{perf}(Y)^*$ is a
{\it relative} integral  functor of kernel $\omega$.
\end{defn}

\begin{ex} If $\omega$ is exact and perfect, then $\omega\simeq \omega_K$ for
a unique $K\in D^b_c(X\times_T Y)$, $\Phi_\omega$ takes values in
$D(Y)^\vee\simeq D^b_c(Y)$ and $\Phi_\omega\simeq\Phi_K$.
\end{ex}

The extension theorem has now  the following form:

\begin{thm}\label{lifting-rel} Let $p\colon X\to T$ and $q\colon Y\to T$ be two proper and flat $T$-schemes.
Let $F\colon D_\text{\rm perf}(X)\to D_\text{\rm perf}(Y)^*$ be a
$T$-linear functor. For any proper and flat $T$-scheme $f\colon
S\to T$ there exists a functor
\[ F_S\colon D_{\text{\rm fhd}/X,S}(X\times_T S)\to D_\text{\rm perf}(Y\times_T S)^*\] such that:

1) $F_S$ is $S$-linear: one has a bi-functorial isomorphism
$F_S(M\otimes p^*N)\simeq F_S(M)\otimes q^*N$, for any $M\in
D_{\text{\rm fhd}/X,S}(X\times_T S)$, $N\in D_\text{\rm perf}(
S)$.

2) It is compatible with $F$: for any $M\in D_{\text{\rm
fhd}/X,S}(X\times_T S)$ one has a natural isomorphism
$f_*F_S(M)\simeq F(f_*M)$.

3) $F_S$ is functorial on $F$.
\end{thm}

\begin{proof} The proof is completely analogous to that of Theorem
\ref{lifting}. One constructs the relative version of the
``perfect resolution'' of section \ref{section-resolution}, just
replacing $k$ by $T$. The $T$-linearity of $F$ is necessary to
reproduce Lemma \ref{lemita} in the relative setting. The flatness
hypothesis is necessary for the use of flat base change and 1) of
Proposition \ref{fhd-properties}.
\end{proof}

As in the absolute case, we obtain corollaries:

\begin{cor} One has a $T$-linear functor
\[\aligned \Psi\colon \bHom_T (D_\text{\rm perf}(X),D_\text{\rm perf}(Y)^*)  &\to D_\text{\rm perf}
(X\times_T Y)^*\\ F&\mapsto F_X(\cO_\Delta)\endaligned  \] and the
composition $\Phi\circ\Psi$ is isomorphic to the identity.
\end{cor}

\begin{cor} Assume that $X$, $Y$ and $T$ are projective $k$-schemes. One has $T$-linear functors
\[ \aligned \Phi\colon D_\text{\rm perf}(X\times_T Y)^{\text{\rm bi-}\vee} &\to \bHom_T^{\text{\rm ex}}
(D_\text{\rm perf}(X), D^b_c(Y) )  \\
\omega &\mapsto \Phi_\omega
\endaligned $$ and $$
\aligned \Psi\colon \bHom_T^{\text{\rm ex}}(D_\text{\rm
perf}(X),D^b_c(Y)) &\to
D_\text{\rm perf} (X\times_T Y)^{\text{\rm bi-}\vee}\\
F&\mapsto F_X(\cO_\Delta)\endaligned  \] and the composition
$\Psi\circ\Phi$ is isomorphic to the identity.
\end{cor}

\bibliographystyle{siam}

\begin{thebibliography}{1}

\bibitem{Ballard-1} {\sc M.R.~Ballard}, {\em Derived categories of sheaves on quasi-projective schemes},
arXiv:0801.2599v2.

\bibitem{BoLarLu}
{\sc A.~Bondal. M.~Larsen, V.~Lunts}, {\em Grothendieck ring of
pretriangulated categories}, Int. Math. Res. Not., 29 (2004),
pp~1461--1495.

\bibitem{BoVdB03}
{\sc A.~Bondal and M.~Van~den Bergh}, {\em Generators and
representability of  functors in commutative and noncommutative
geometry}, Mosc. Math. J., 3  (2003), pp.~1--36.


\bibitem{Cal}
{\sc A. Caldararu}, {\em Derived categories of sheaves:
askimming}, Snowbird lectures in algebraic geometry, 43--75,
Contemp. Math. 388, Amer. Math. Soc. Providence RI (2005).

\bibitem{CaSte07}
{\sc A.~Canonaco and P.~Stellari}, {\em Twisted fourier-mukai
functors}, Adv.
  Math., 212 (2007), pp.~484--503.

\bibitem{CaSte10}
{\sc A.~Canonaco and P.~Stellari}, {\em Fourier-Mukai functors in
the supported case }, arXiv 1010.0798.

\bibitem{CaSte10-2}
{\sc A.~Canonaco and P.~Stellari}, {\em Non-uniqueness of
Fourier-Mukai kernels}, arXiv 1009.5577.

\bibitem{HLS05}
{\sc D.~Hern{\'a}ndez Ruip{\'e}rez, A.C.~L{\'o}pez Mart{\'{\i}}n
and F.~Sancho de Salas}, {\em Fourier-{M}ukai transforms for
{G}orenstein Schemes}, Adv. in Maths., 211 (2007), pp.~594--620.

\bibitem{LuOr}
{\sc V. Lunts, D.~O. Orlov}, {\em Uniqueness of enhancements for
triangulated categories}, J. Amer. Math. Soc, 23 (2010),
pp.~853--908.

\bibitem{Kaw02}
{\sc Y.~Kawamata}, {\em Equivalences of derived categories of
sheaves on smooth
  stacks}, Amer. J. Math., 126 (2004), pp.~1057--1083.



\bibitem{Or97}
{\sc D.~O. Orlov}, {\em Equivalences of derived categories and
${K}3$
  surfaces}, J. Math. Sci. (New York), 84 (1997), pp.~1361--1381.
\newblock Algebraic geometry, 7.

\bibitem{SaSa}
{\sc C.~Sancho de Salas and F.~Sancho de Salas}, {\em The linear
dual of the derived category of a scheme}, Proc. Amer. Math. Soc.
(to appear).

\bibitem{To07}
{\sc B. T\"{o}en}, {\em The homology theory of dg-categories and
derived Morita theory}, Invent. Math, 167 (2007), pp.~615--667.



\end{thebibliography}

\end{document}